\documentclass[conference]{IEEEtran}
\IEEEoverridecommandlockouts

\usepackage{algorithm}
\usepackage[noend]{algpseudocode}

\usepackage[dvipsnames]{xcolor}
 
\usepackage{tikz}
\usetikzlibrary{positioning}
\usepackage{amsmath}
\usepackage{amssymb}
\usepackage{amsthm}
\usepackage{amsfonts}
\usepackage{url}
\usepackage{nccmath}
\newtheorem{theorem}{Theorem}
\usepackage{svg}

\newtheorem{lemma}{Lemma}
\newtheorem{proposition}{Proposition}

\newtheorem{algo}{Algorithm}
\theoremstyle{definition}

\newtheorem{remark}{Remark}
\usepackage{mathtools}
\usepackage[margin=2cm]{geometry}
\usepackage{graphicx}
\graphicspath{ {./images/} }
\usepackage{subcaption}
\usepackage{color}
\usepackage{xcolor}
\usepackage{float}
\usepackage[noadjust]{cite}
\usepackage{balance}
\usepackage[noabbrev]{cleveref}
\crefformat{equation}{(#1)}
\usepackage{bbm}
\usepackage{bm}
\usepackage{comment}
\usepackage{arcs}

\usepackage[normalem]{ulem}

\usepackage{resizegather}

\usepackage{acronym}
\newacro{PO}[PO]{Policy Optimization}
\newacro{pg}[PG]{Projected Gradient}
\newacro{ouralgo}[NOA]{Name of Our Algo}
\newacro{rcm}[RCM]{Riemannian Center of Mass}
\newacro{gcvx}[g-convex]{geodesically convex}

\newcommand{\R}{\mathbb{R}}
\newcommand{\A}{\mathcal{A}}
\newcommand{\tr}{\mathrm{tr}}
\newcommand{\RCM}{\text{RCM}}
\newcommand{\C}{\mathcal{C}}

\newcommand{\inj}{\mathrm{inj}}
\newcommand{\bbG}{G}
\newcommand{\calG}{\mathcal{G}}
\newcommand{\frakg}{\mathfrak{g}}
\newcommand{\Span}{\mathrm{span}}
\newcommand{\col}{\mathrm{col}}
\newcommand{\Log}{\mathrm{Log}}
\newcommand{\mat}[1]{\begin{bmatrix} #1 \end{bmatrix}}
\newcommand{\calM}{\mathcal{M}}

\usepackage{cite}
\usepackage{textcomp}
\def\BibTeX{{\rm B\kern-.05em{\sc i\kern-.025em b}\kern-.08em
		T\kern-.1667em\lower.7ex\hbox{E}\kern-.125emX}}
\usepackage{enumerate}
\usepackage[inline]{enumitem}

\def\x{{\bm x}}
\def\u{{\bm u}}
\def\z{{\bm z}}
\def\v{{\bm v}}
\def\w{{\bm w}}

\def\e{{\bm e}}

\def\1{{\bm 1}}

\def\0{{\bm 0}}

\def\Exp{\mathrm{Exp}}

\makeatletter
\newcommand{\algorithmfootnote}[2][\footnotesize]{%
  \let\old@algocf@finish\@algocf@finish
  \def\@algocf@finish{\old@algocf@finish
    \leavevmode\rlap{\begin{minipage}{\linewidth}
    #1#2
    \end{minipage}}%
  }%
}
\makeatother

\title{Consensus on Lie groups for\\the Riemannian Center of Mass}

\author{Spencer Kraisler, \textit{Student Member, IEEE,} Shahriar Talebi, \textit{Student Member, IEEE,}\\ and Mehran Mesbahi, \textit{Fellow, IEEE} 

\thanks{The authors are with the William E. Boeing Department of Aeronautics and Astronautics, University of Washington, Seattle, WA, USA. Emails: {\em\small\{kraisler+shahriar+mesbahi\}@uw.edu}. The research of the authors has been supported by {NSF grant ECCS-2149470 and AFOSR grant FA9550-20-1-0053.}}
}

\begin{document}

\tikzstyle{block} = [draw, fill=blue!20, rectangle, 
    minimum height=3em, minimum width=6em]
\tikzstyle{sum} = [draw, fill=blue!20, circle, node distance=1cm]
\tikzstyle{input} = [coordinate]
\tikzstyle{output} = [coordinate]
\tikzstyle{pinstyle} = [pin edge={to-,thin,black}]

\maketitle

\begin{abstract}
In this paper, we develop a consensus algorithm for distributed computation of the \ac{rcm} on Lie Groups. The algorithm is built upon a distributed optimization reformulation that allows developing an intrinsic, distributed (without relying on a consensus subroutine), and a computationally efficient protocol for the \ac{rcm} computation. The novel idea for developing this fast distributed algorithm is to utilize a Riemannian version of distributed gradient flow combined with a gradient tracking technique. We first guarantee that, under certain conditions, the limit point of our algorithm is the \ac{rcm} point of interest. We then provide a proof of global convergence in the Euclidean setting, that can be viewed as a ``geometric'' dynamic consensus that converges to the average from arbitrary initial points. Finally, we proceed to showcase the superior convergence properties of the proposed approach as compared with other classes of consensus optimization-based algorithms for the \ac{rcm} computation.
\end{abstract}
\begin{IEEEkeywords}
	    \textit{Consensus on Lie Groups, Riemannian Center of Mass, Karcher Mean, Multiagent Systems}
\end{IEEEkeywords}

\section{Introduction}
Consensus algorithms are a ubiquitous class of protocols with pertinent applications to fields such as distributed estimation, optimization, and control of multi-agent systems. At a foundational level, consensus algorithms steer a set of dynamical agents towards a single point. \textit{Average} consensus necessitates this point to be the average of the agents' initial states. Statistical and computational advantages of average consensus algorithms have found a wide range of applications in distributed resource allocation, formation control, and distributed estimation \cite{olfati2007consensus,strogatz2012sync,kia2019tutorial}. 

While consensus algorithms have predominantly been studied for Euclidean spaces, there has been a number of efforts in the literature to generalize this protocol to Riemannian manifolds \cite{tron2012riemannian,olfati2006swarms,sarlette2010coordinated,sepulchre2010consensus,Chen2014-ds,Chen2014-ft,kraisler2023distributed}. One notable application is 3D localization of camera sensors \cite{tron2014distributed}. In this direction, useful notions from Euclidean geometry pertinent to the notion of average consensus generalize to Riemannian manifolds. This includes the \acf{rcm}, replicating the notion of an average \cite{afsari2011riemannian}.



The Riemannian consensus algorithms discussed so far achieve consensus but not \textit{consensus to the \ac{rcm}} (or the so-called ``\ac{rcm} consensus''). To the best of our knowledge, \cite[\S 3.2]{tron2008distributed} is the earliest work on the topic of \ac{rcm} consensus. The method was later extended to a larger regime of Riemannian manifolds in \cite{tron2011average}. In order to guarantee \ac{rcm} consensus, this approach utilizes a consensus subroutine within each iteration. However, practical implementation of this subroutine may not be always favorable and time-efficient. 
This calls for a fast distributed \ac{rcm} consensus algorithm that is,
\begin{enumerate*}
    \item intrinsic, and as such, parameterization of the manifold does not affect its properties;
    \item completely distributed without relying a consensus subroutine; and
    \item has at least a linear convergence rate.
\end{enumerate*} 

In this paper, we first reformulate the \ac{rcm} for a set of points on a Riemannian manifold as the unique optimal solution to a consensus optimization problem. This perspective calls to explore an array of techniques including distributed optimization on Riemannian manifolds and dynamic average consensus. In this direction, we also provide an optimizer for Lie groups equipped with a bi-invariant metric with the aforementioned properties, thus achieving a distributed \ac{rcm} consensus algorithm. We prove that, under certain conditions, the limit point of our algorithm is the \ac{rcm} of the original points. We provide global convergence guarantees in the Euclidean case and compare the performance of our proposed algorithm with distributed constrained optimization approaches such as penalty-based and Lagrangian-based methods.

The rest of the paper is organized as follows. \S\ref{sect:background} offers the problem formulation and a technical background on Riemannian geometry, \ac{rcm} and Lie groups. In \S\ref{sect:contribution}, we present the RCM optimization reformulation utilized to develop fast and simple distributed \ac{rcm} consensus. We then proceed with limit point analysis and global convergence guarantees of the proposed algorithm for the Euclidean space in \S\ref{sect:results}. Finally, several simulation scenarios are presented in \S\ref{sec:simulation}, followed by concluding remarks in \S\ref{sec:conclusion}.




\section{Problem Statement and Background}\label{sect:background}

Before delving deeper into necessary mathematical background, let us first present the problem statement. Consider a Lie group $\calG$ equipped with a bi-invariant Riemannian metric and a communication network of agents $[N] = \{1,2,\ldots, N\}$ represented by an undirected connected graph $\bbG = ([N],E)$, were $E$ represents the edge set. Given $N$ points $z_i \in \calG$, the goal is to design a distributed (dynamic) protocol (with respect to $\bbG$) for each agent as $\dot{x}_i(t) = \mathbf{F}_i(t,x_i, x_j:j \sim i)$, steering the agents' states to the \ac{rcm} of $\{z_1,...,z_N\}$. 

\subsection{Riemannian Geometry}
For geometric concepts, we follow the standard notation as in \cite{lee2018introduction}. Let $(\calM,\langle .,. \rangle)$ be a Riemannian manifold and $\mathrm{d}(.,.)$ the corresponding induced geodesic distance. In this paper, we assume that $\calM$ has  a bounded sectional curvature. We denote its N-fold Cartesian product as $\calM^N$ and use boldface letter $\x=(x_1,\cdots,x_N) \in \calM^N$ to denote a point on the product space. 

We denote the open \textit{geodesic ball} centered at $x \in \calM$ with radius $r > 0$ as,
\begin{equation*}
    \mathcal{B}(x, r) := \{y \in \calM: \mathrm{d}(x,y) < r\}.
\end{equation*}
A subset $A \subset \calM$ is called \acf{gcvx} if for any $x,y \in A$, there exists a unique \textit{minimizing} geodesic $\gamma:[0,T] \to \calM$ connecting $x$ to $y$ such that $\gamma$ is contained in $A$ and there exists no other \textit{minimizing} geodesic connecting the pair.\footnote{Some authors call this \textit{strong} geodesic convexity.} A function $f:A \to \R$ is called \ac{gcvx} if for any geodesic $\gamma$ contained in $A$, $f \circ \gamma$ is convex in the Euclidean sense \cite{boumal2020introduction}. The \textit{convexity radius} of $\calM$ is defined as
\begin{equation*}
    r^* := \frac{1}{2} \min(\inj( \calM), \frac{\pi}{\sqrt{\Delta}}),
\end{equation*} where $\inj (\calM)$ is the injectivity radius and $\Delta$ is the upper bound on the sectional curvature of $\calM$. If $r \leq r^*$, then $\mathcal{B}(x,r)$ is \ac{gcvx} \cite{afsari2011riemannian, tron2012riemannian}.

Next, for $\z=(z_1,\dots,z_N) \in \calM^N$, define its \textit{\acf{rcm}} as any minimizer of,
\begin{equation}\label{rcm-problem}
    \min_{y \in \calM}\; \sum_{i=1}^N \mathrm{d}(y,z_i)^2 ;
\end{equation} the \ac{rcm} may not exist nor be unique. We denote the \ac{rcm} of $\z$ as $\RCM(\z)$ or $\bar{z}$ when it exists and is unique. One can induce existence and uniqueness guarantees as follows. Define the \textit{convexity submanifold}:
\begin{equation}\label{convexity-submanifold}
    \C:=\{\z \in \calM^N: \exists y \in \calM,  r < r^* \ \text{s.t. } \  z_i \in \mathcal{B}(y,r) \ \forall i\}.
\end{equation} If $\z \in \C$, then $\RCM(\z)$ exists and is unique \cite[Thm. 2.1]{afsari2011riemannian}. 
Now, let $\mathcal{B}$ be any geodesic ball satisfying the condition in \cref{convexity-submanifold} for $\z$. If  some $z \in \mathcal{B}$ satisfies the \textit{Karcher equation},
\begin{equation}\label{eq:karcher}
    \sum_{i=1}^N \log_{z}(z_i)=0,
\end{equation} then we necessarily have $z = \RCM(\z)$. 

\subsection{Lie Groups}
A \textit{Lie group} $\calG$ is a smooth manifold with a group structure where the group and inverse operators are smooth mappings. Every (Lie) group admits the identity element $e$ such that $xe = ex = x$ and $x x^{-1} = x^{-1}x = e$ for all $ x \in \calG$. The tangent space of $\calG$ at the identity is called the \textit{Lie algebra} $\frakg := T_e\calG$. The Lie algebra is a vector space equipped with a vector multiplication operator called the commutator. We suggest \cite{arvanitogeorgos2003introduction,sachkov2009control} for introductions to Lie groups with emphasis on Riemannian geometry and control theory. 

A Riemannian metric $\langle .,. \rangle$ on $\calG$ is called \textit{left-invariant} if 
\begin{equation*}
    \langle \xi, \eta \rangle_e = \langle dL_x \xi , dL_x \eta \rangle_x,
\end{equation*} for any $x \in \calG$, $\xi, \eta \in \frakg$. Here, $dL_x$ is the differential of the left-translation map $L_x:y \mapsto xy$. Right-invariant metrics are defined similarly with the right-translation map $R_x:y \mapsto yx$. Bi-invariant metrics are both left- and right-invariant. If $\calG \subset GL(n)$ is a \textit{matrix Lie group}, then $dL_x \xi = x\xi$. 

All Lie groups admit left- and right-invariant metrics \cite[Lemma 3.10]{lee2018introduction}. A Lie group admits a bi-invariant metric if and only if it is isomorphic to the Cartesian product of a compact Lie group and a vector space \cite[Lemma 7.5]{milnor1976curvatures}. An example of a Lie group with a bi-invariant metric is the unit quaternions equipped with the dot product. 

The manifold and Lie group exponential coincide at the identity when the metric is bi-invariant \cite[Prop. 3.10.]{arvanitogeorgos2003introduction}. Since left-translation is an isometry, the following identities hold for these Lie groups \cite[Prop. 5.9]{lee2018introduction}: for any $x,y \in \calG$ and $\xi \in \frakg$,
\begin{subequations}
    \begin{align*}
        \exp_x(dL_x \xi) &= L_x(\Exp(\xi)), \\
        \log_x(y) &= dL_x\Log(x^{-1}y), 
    \end{align*}
\end{subequations} 
where $\Exp$ and $\exp_x$ are, respectively, the Lie and manifold exponential at $x$; $\Log$ and $\log_x$, on the other hand, are their corresponding inverse mappings (wherever they are well-defined). This also implies that on $T_x\calG$, we have
\begin{equation*}
    \nabla_x \left[\frac{1}{2}\mathrm{d}(x,y)^2\right] = -\log_{x}(y) = -dL_x \Log(x^{-1}y).
\end{equation*}


\section{The Algorithm and its Derivation}\label{sect:contribution}
In this section, we present a reformulation of \cref{rcm-problem} and propose a solution algorithm that is intrinsic, distributed, and has an (empirically) linear convergence rate. 

\subsection{\ac{rcm} Consensus Reformulation}
Using a consensus reformulation, we can reformulate \cref{rcm-problem} as a \ac{gcvx} consensus optimization problem,
\begin{subequations}\label{dist-rcm-problem}
    \begin{align}
        \min_{\x \in \calM^N} & \mathbf{f}(\x) := \sum_{i=1}^N \mathrm{d}(x_i,z_i)^2 \\
        \text{s.t. } & \; \x \in \A, \label{constraint}
    \end{align}
\end{subequations} where $\A=\{(x,...,x): x \in \calM\}$ is the \textit{agreement submanifold}. 
Also, note that on a connected graph, \cref{constraint} is equivalent to zero \textit{consensus error} $\varphi(\x) = 0$, defined as
\begin{equation}\label{consensus-error}
    \varphi(\x) := \frac{1}{2} \sum_{\{i,j\} \in E} \mathrm{d}(x_i,x_j)^2 = \frac{1}{4} \sum_{i=1}^N \sum_{j \sim i} \mathrm{d}(x_i,x_j)^2.
\end{equation}
Thus, the unique solution to \cref{dist-rcm-problem} on $\C$ (where its existence and uniqueness is guaranteed) corresponds to
$\x^* := (\bar{z},...,\bar{z})$.

\begin{figure*}[tp]
    \centering
    \begin{subfigure}{0.48\linewidth}
        \centering
        \includegraphics[trim={.7cm 0 .1cm 0},clip,width=.7\columnwidth]{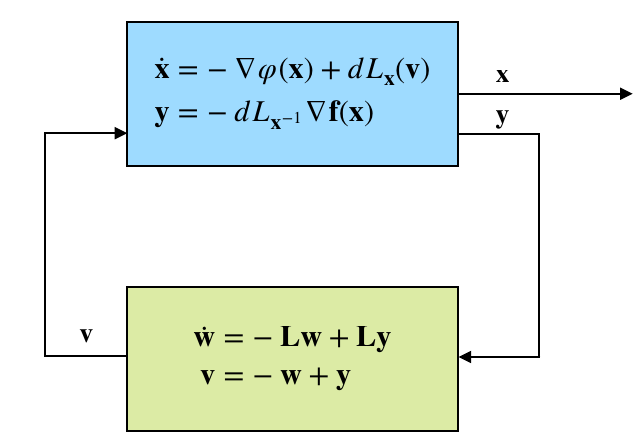}
        \vspace{0.7cm}
        \caption{}
        \label{fig:block-diagram}
    \end{subfigure}
    \hfill
    \begin{subfigure}{0.5\linewidth}
        \centering
        \includegraphics[width=\columnwidth]{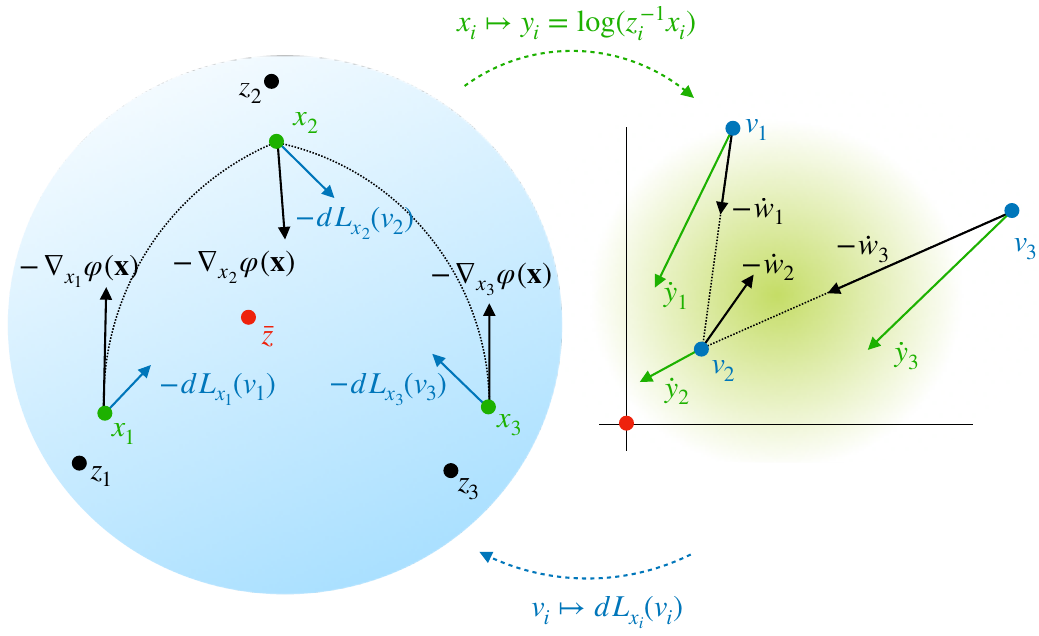}
        \caption{}
        \label{fig:Lie-correspondence}
    \end{subfigure}
    \caption{ (a) Block diagram of Algorithm \ref{algo:distributed-solver}. (b) Illustration of the correspondence in Algorithm \ref{algo:distributed-solver} between the state consensus on the Lie group (left) and the gradient dynamic consensus on its Lie algebra (right), interconnected by two mappings.}
    \vspace{-0.4cm}
\end{figure*}

This optimization reformulation enables us to design a distributed algorithm for \ac{rcm} consensus using \textit{local first order information}. That is, for agent $i$ we use,\footnote{These closed-form expressions only hold for Lie groups with bi-invariant metrics. Nonetheless, \cref{dist-rcm-problem} is still distributed for arbitrary Riemannian manifolds but \cref{closed-form-expressions} will be different.}
\begin{subequations}\label{closed-form-expressions}
    \begin{align}
        \nabla_{x_i} \mathbf{f}(\x) &= -dL_{x_i} \Log(x_i^{-1} z_i), \\
        \nabla_{x_i} \varphi(\x)  &=\textstyle -dL_{x_i}\sum_{j \sim i} \Log(x_i^{-1}x_j),
    \end{align}
\end{subequations} 
which depend only on the variables $\{x_i, x_j:j \sim i\}$ that are locally available. Therefore, a fast first order optimizer will be a solution to our problem.

\subsection{Our Algorithm}\label{sect:solver}

\begin{algo}\label{algo:distributed-solver} 
    Given $\z = (z_1,\cdots,z_N) \in \C \subset \calG^N$, let $\mathbf{L}:\frakg^N \to \frakg^N$ be the graph Laplacian of $\bbG$, represented as a linear operator on the product Lie algebra $\frakg^N$. For each agent $i$, set $x_i(0):=z_i \in \calG$ and $w_i(0):=0 \in \frakg$. The proposed algorithm assumes the following dynamics:
    \begin{subequations}\label{solver}
        \begin{align}
            \dot{\x} &= -\nabla \varphi(\x) - dL_{\x} \v \label{first-line}\\
            \dot{\w} &= \mathbf{L} \v,
        \end{align}
    \end{subequations}
    where $\v = -\w + dL_{\x^{-1}} \nabla \mathbf{f}(\x)$.
\end{algo} 

In order to show that Algorithm \ref{algo:distributed-solver} is fully distributed, notice that the dynamics for each agent $i$ reduces to
\begin{align*}
    \dot{x}_i &= 
    dL_{x_i}\big(\textstyle \sum_{j \sim i} \Log(x_i^{-1}x_j) - v_i \big), \\ 
    \dot{w}_i &= \textstyle \sum_{j \sim i} (v_i - v_j),
\end{align*} 
with $v_i = -w_i +  \Log(z_i^{-1}x_i)$.
Later, we show that each $v_i \in \frakg$ is tracking a global information about the distributed cost $\mathbf{f}$. Also, $w_i$ is a latent state for a dynamic consensus algorithm on the Lie algebra \cref{gradient-tracking}.

\subsection{Derivation of our Algorithm}
Similar to the gradient flow as the simplest (continuous) optimizer for an optimization problem, the Distributed Gradient Flow (DGF) can be viewed as the simplest optimizer for a \textit{distributed} optimization problem. The state dynamics for DGF assumes the form,
\begin{equation}\label{DGF}
    \dot{\x} = -\nabla \varphi(\x) - \nabla \mathbf{f}(\x).
\end{equation} Locally, each agent's state follows the dynamics 
\begin{equation*}
    \textstyle \dot{x}_i = dL_{x_i} \sum_{j \sim i} \Log(x_i^{-1}x_j) + dL_{x_i} \Log(x_i^{-1}z_i).
\end{equation*}
The first forcing term above drives the consensus error $\varphi$ to zero, whereas the second term attempts to minimize the \textit{local cost} $x \mapsto \frac{1}{2} \mathrm{d}(x,z_i)^2$. This is a first attempt at minimizing the cost $x \mapsto \mathbf{f}(x,...,x) = \sum_{i=1}^N \frac{1}{2} \mathrm{d}(x,z_i)^2$ in a distributed manner, whose global minimizer is $x=\bar{z}$. However, this method easily fails due to the excess of points that satisfy $\nabla \varphi(\x) = -\nabla \mathbf{f}(\x)$ with $\x \not \in \A$. 

One can think of DGF as an ``open loop'' distributed optimizer for \Cref{dist-rcm-problem}, i.e., while \Cref{DGF} includes global state feedback with the consensus term, there is no global \textit{cost gradient feedback}. To resolve this issue, each agent needs access to the global information about the average gradients of local costs of all agents. The simplest way to introduce this cost gradient feedback is through Gradient Tracking (GT) \cite{nedic2017achieving}. This involves implementing a dynamic consensus algorithm on the Lie algebra for tracking this average gradient:
\begin{subequations}\label{gradient-tracking}
    \begin{align}
        \dot{\mathbf{w}} &= -\mathbf{L}\mathbf{w} + \mathbf{L}\mathbf{u}, \\
        \mathbf{v} &= -\mathbf{w} + \mathbf{u}.
    \end{align}
\end{subequations} Here, $w_i(0)=0$ and we set $u_i = dL_{x_i^{-1}}\nabla_{x_i}\mathbf{f}(\x)$. Under these dynamics, each $v_i$ will track the same global information $\frac{1}{N}\sum_{j=1}^N dL_{x_j^{-1}}\nabla_{x_j} \mathbf{f}(\x)$. 

Finally, combining DGF with GT, which is achieved by replacing the term $-\nabla \mathbf{f}(\x)$ in \cref{DGF} with $-dL_{\x} \v$, gives us \Cref{solver}.  In other words, we are ``closing the loop'' on the DGF by introducing a cost gradient feedback for each agent using GT. A schematic diagram of this design is also depicted in \Cref{fig:block-diagram} with its geometric interpretation in \Cref{fig:Lie-correspondence}.

\begin{remark}
    We do no require $z_i$'s to be constant. Suppose each $z_i$ is time-varying and we wish for each $x_i$ to track the time-varying RCM in a distributed manner. Indeed, we have observed that our algorithm works empirically in this dynamic consensus setup. 

\end{remark}

\begin{remark}
    Note that we have initialized $\x(0)=z_i$. This can be relaxed to only requiring $\x(0) \in \C$. The ``closer'' the points are initialized to $\bar{z}$, the faster the convergence. However, since $\bar{z}$ is unknown a priori, we simply chose $\x_i(0)=z_i$; another candidate would be $\x_i(0) = e$. 
\end{remark}

\section{Main Results}\label{sect:results}
In this section, we first provide the limit point analysis of Algorithm \ref{algo:distributed-solver} in the general case and then guarantee convergence for the Euclidean case. 

\begin{proposition}[Limit point]\label{prop:stationary-point}
    Let $\mathcal{B} \subset \calG$ be a geodesic ball with radius $r < r^*$ containing $\{z_1,...,z_N\}$, and consider the trajectory $(\x(.),\w(.))$ generated by Algorithm \ref{algo:distributed-solver}. Suppose $(\x(.),\w(.))$ converges to some $(\x^*,\w^*)$ such that $\x^* \in \mathcal{B}^N$. Then $x_i^* = \bar{z}$ and $w_i^* = \Log(z_i^{-1} \bar{z})$ for each $i$.
\end{proposition}
\begin{proof}
    We can compactly express Algorithm \ref{algo:distributed-solver} as
    \begin{equation*}
        \mat{\dot{\x} \\ \dot{\w}} = \mat{-\nabla \varphi(\x) - dL_\x [-\w + \Log(\z^{-1}\x)\\ \mathbf{L}[-\w + \Log(\z^{-1}\x)] }.
    \end{equation*}
    Since we assume $(\x^*,\w^*)$ is a fixed point,
    \begin{equation*}
        \mathbf{L}\left[-\w^* + \Log(\z^{-1}\x^*)\right]=0,
    \end{equation*} and hence,
    \begin{equation*}
        -\w^* + \Log(\z^{-1}\x^*) = (\xi,...,\xi) =: \xi \1 \in \frakg^N ,
    \end{equation*} for some $\xi \in \frakg$. Hence, $dL_{(\x^*)^{-1}} \nabla \varphi(\x^*) = \xi \1$. Then
    \begin{equation*}
          N \|\xi\|^2 = \langle \xi \1, \xi \1 \rangle = \langle dL_{(\x^*)^{-1}} \nabla \varphi(\x^*), \xi \1 \rangle .
    \end{equation*} 
    Since the metric is bi-invariant, we have
    \begin{gather*}
        dL_{(x_i^*)^{-1}} \Log_{x_i^*}(x_j^*)  = \Log((x_i^*)^{-1}x_j^*) =-dL_{(x_j^*)^{-1}} \Log_{x_j^*}(x_i^*) 
    \end{gather*} implying
    $\textstyle \sum_{i=1}^N dL_{(x_i^*)^{-1} }\sum_{j \sim i} \log_{x_i^*}(x_j^*)=0$. Therefore
    \begin{gather*}
        \langle dL_{(\x^*)^{-1}} \nabla \varphi(\x^*), \xi \1 \rangle = \left\langle \sum_{i=1}^N \sum_{j \sim i} \Log((x_i^*)^{-1}x_j^*), \xi \right\rangle =0.
    \end{gather*} Thus we must have $\xi = 0$, and therefore $\nabla \varphi(\x^*)=0$.

    Next, it is shown in \cite[Theorem 5]{tron2012riemannian} that the critical points of $\varphi$ restricted to $\C$ coincide with $\A$. So, $\x^* \in \C$ and $\nabla \varphi(\x^*)=0$ together imply $\x^* \in \A$. Next, note that
    \begin{equation*}
        \textstyle\frac{d}{dt} \sum_{i=1}^N\dot{w}_i(t) = 0.
    \end{equation*} 
    As such, since $\w(0) = 0$, we must have $\sum_{i=1}^N w_i^* = 0$. Therefore, by the fact that $-\w^* + \Log(\z^{-1}\x^*) = 0$, we arrive at
    \begin{equation*}
        \textstyle\sum_{i=1}^N \Log(z_i^{-1}x^*)=0.
    \end{equation*} Since $x^*,z_1,...,z_N \in \mathcal{B}$ and $x^*$ satisfies \Cref{eq:karcher}, $x^* = \bar{z}$, and thus $x_i^* = \bar{z}$ and $w_i^* = \Log(z_i^{-1}\bar{z})$. 
\end{proof}

Next, we provide the convergence guarantees of our algorithm for $\calG=\R^n$; the corresponding analysis for arbitrary Lie groups is the subject of our future work.
Herein, we write $\x = \col(x_1,...,x_N) \in \R^{Nn}$ to denote vertical concatenation and use similar notation for $\z$, $\x$, and $\u$. Then the dynamics of Algorithm \ref{algo:distributed-solver} reduces to the following matrix form:
\begin{subequations}\label{euclidean-solver}
    \begin{align}
        \dot{\x} &= -(\mathbf{L}\otimes I_n) \x + \w - (\x - \z) \\
        \dot{\w} &= -(\mathbf{L}\otimes I_n)\w + (\mathbf{L}\otimes I_n)(\x - \z),
    \end{align} 
\end{subequations}
where $\otimes$ denote the Kronecker product. The \ac{rcm} of initial states $\z$ reduces to the Euclidean average
\begin{equation*}
    \textstyle \bar{z} = \RCM(\z)=\frac{1}{N}\sum_{i=1}^N z_i \in \R^n,
\end{equation*}
and so $\x^* = \col(\bar{z},...,\bar{z})=\1_N \otimes \bar{z}$. The next result establishes guaranteed average consensus  starting from arbitrary initial points.

\begin{theorem}[Convergence in $\R^n$]\label{thm:conv-linear}
    Suppose $\x(t),\w(t)$ is the trajectory generated by \cref{euclidean-solver} over a connected graph $\bbG$ with $\w(0)=0$. Then, starting from any arbitrary starting point $\x(0)$, we have an exponentially convergent trajectory with $\lim_{t \to \infty} x_i(t) = \bar{z}$ and $\lim_{t \to \infty} w_i(t) = \bar{z} - z_i$ for all $i$.
\end{theorem}

\begin{proof}
First, by noting that $\v = -\w + \x - \z$ we reformulate the system in \cref{euclidean-solver} as
    \begin{equation}
        \mat{\dot{\v} \\ \dot{\x}} =  (A \otimes I_n) \mat{\v \\ \x},
    \end{equation}
    where
    \begin{equation}\label{eq:A-dyn}
        A \coloneqq \mat{-\mathbf{L} - I_N & -\mathbf{L} \\ -I_N & -\mathbf{L}}.
    \end{equation}
    Before proceeding, we need the following results on characterizing the spectrum of $A$ with its proof deferred to the end of this section.
\begin{lemma}\label{lem:A-spec}
Suppose $\bbG$ is connected. Then zero is a simple eigenvalue of the matrix $A$ in (\ref{eq:A-dyn}). Furthermore, all non-zero eigenvalues of this matrix have negative real parts.
\end{lemma}
\noindent The spectral properties of $A$ established above then implies that the dynamics of \cref{eq:A-dyn} is marginally stable. 
Let $p,q$ be, respectively, the right and left eigenvectors of $A$ associated with $\lambda=0$ such that $p^\top q=1$. So, $p \in \mathcal{N}(A)$. Also, since 0 is a simple eigenvalue of $A^\top$, we get
\begin{equation*}
        \mathcal{N}(A^\top)=\Span \left(\mat{\1_N \\ -\1_N} \right).
\end{equation*}
Thus, we set $p = \mat{0_N \\ \1_N}$, $q = \frac{1}{N} \mat{-\1_N \\ \1_N}$. By properties of Kronecker product, we obtain that $A \otimes I_n$ is also marginally stable with a zero eigenvalue of algebraic and geometric multiplicity $n$. This eigenvalue has corresponding right eigenvectors $p \otimes \e_i$ and left eigenvectors $q \otimes \e_i$. Here, $\e_i \in \R^n$ is the $i$-th standard basis for $i=1,2,\cdots,n$. 

Next, by computing $e^{(A\otimes I_n)t}$ using the Jordan decomposition and taking the limit as $t \to \infty$ we obtain that (see \cite[Proposition 3.11]{mesbahi2010graph} for a similar computation)
\begin{equation*}
    \lim_{t\to \infty}e^{(A\otimes I_n)t} = \sum_{i=1}^n (p \otimes \e_i) (q \otimes \e_i)^\top = (p q^\top) \otimes I_n.
\end{equation*}
Therefore, \Cref{eq:A-dyn} converges as follows
\begin{equation*}
    \mat{\v^* \\ \x^*} \coloneqq \lim_{t \to \infty}\mat{\v(t) \\ \x(t)} = ((p q^\top) \otimes I_n) \mat{\v(0) \\ \x(0)},
\end{equation*}  
where $\w(0) = \0_{Nn}$ and $\x(0)$ is the arbitrary starting point. Thus, $\v(0) = -\w(0) + \x(0) - \z = \x(0) - \z$, and
\begin{subequations}
    \begin{align*}
        \mat{\v^* \\ \x^*} &= \frac{1}{N} \left( \mat{\0_{N\times N} \quad \0_{N\times N}\\ -\1_N \1_N^\top \quad \1_N \1_N^\top } \otimes I_n \right) \mat{\x(0) - \z \\ \x(0)} \\
        &=\mat{\0_{Nn} \\ \frac{1}{N} (\1_N \1_N^\top \otimes I_n) \z}= \mat{\0_{Nn} \\ \1_N \otimes \bar{z}}.
    \end{align*}
\end{subequations} 
Therefore $\x^* = \1_N \otimes \bar{z}$ which completes the proof.
\end{proof}

\begin{proof}[Proof of \Cref{lem:A-spec}]
Let $\mat{\v \\ \x} \in \mathcal{N}(A)$ be non-zero. Then
\begin{subequations}
    \begin{align}
        -\mathbf{L}\v - \v - \mathbf{L}\x &= 0 \label{1a},\\
        -\v - \mathbf{L}\x &= 0\label{1b}.
    \end{align} 
\end{subequations}
By combining the two, we get $\mathbf{L} \v = 0$ implying that $\v \in \Span(\1_N)$. It follows from \cref{1b} that $\mathbf{L} \x = -\v = \alpha \1_N$. Then $\1^\top \mathbf{L} \x = 0 = \alpha \1_N^\top\1_N$, and thus $\alpha = 0$. But this implies that $\v = 0$. Thus $\mathbf{L} \x = 0$, which implies that $\x \in \Span(\1_N)$. Therefore, we can conclude that
\begin{equation*}
    \mathcal{N}(A) = \left\{\mat{\0_N \\ \x}: \x \in \Span(\1_N)\right\}.
\end{equation*} 
Since the nullspace has dimension 1, it follows that 0 is a simple eigenvalue of $A$. 
Next, let $\mat{\v \\ \x}$ be an eigenvector of $A$ associated with a non-zero eigenvalue $\lambda$. For the sake of contradiction, suppose $\Re[\lambda] \geq 0$. Then 
\begin{align*}
\begin{cases}
    \lambda \v &= -\mathbf{L} \v - \v - \mathbf{L}\x \\
    \lambda \x &= -\v - \mathbf{L} \x
\end{cases}
\end{align*}
Note, if $\x = 0$, then $\v=0$. So, in order for $\mat{\v \\ \x}$ to be an eigenvector, we need $\x \not = 0$. By solving for $\x$, we obtain
\begin{equation}
    -\lambda \x = (\lambda I + \mathbf{L})^2 \x.
\end{equation} 
So, $(-\lambda, \x)$  is an eigenpair of the matrix $(\lambda I + \mathbf{L})^2$. The eigenvalues of $(\lambda I + \mathbf{L})^2$ are $\{(\lambda + \rho)^2: \rho \in \sigma(\mathbf{L})\}$. So, there exists some $\rho \in \sigma(\mathbf{L})$ such that $-\lambda = (\lambda + \rho)^2$. Let $\lambda = a + bi$ and note the hypothesis of contradiction is $a \geq 0$ and $a^2 + b^2 \neq 0$. By separating the real and imaginary parts, we obtain
\begin{align*}
    -a &= (a + \rho)^2 - b^2, \\
    -b &= 2(a+\rho)b.
\end{align*}
Suppose $b=0$, then the first equation and the hypothesis $a\geq 0$ implies that $a=0$. But, this implies $\lambda = 0$ which is a contradiction. Next, suppose $b \not = 0$, then the second equation implies $-1 = 2a + 2\rho$. Thus $a = \frac{-1-2\rho}{2}<0$, which is a contradiction with $a \geq 0$ as $\rho \geq 0$.
Therefore, if $\lambda \neq 0$ then we must have $a < 0$, and so all non-zero eigenvalues of $A$ have negative real part. This completes the proof.
\end{proof}

\section{Simulations}\label{sec:simulation}

In this section and for illustration purposes, we consider the special orthogonal Lie group $\calG=SO(3)$, with its parameterization of $3 \times 3$ rotation matrices. The associated Lie algebra, denoted $\mathfrak{so}(3)$, is the space of $3 \times 3$ skew-symmetric matrices. We equip $SO(3)$ with the following bi-invariant Riemannian metric: 
\begin{equation*}
    \langle u,v\rangle_x := \frac{1}{2}\tr(u^\top v) = \frac{1}{2}\tr(\xi^\top \eta),
\end{equation*} where $x \in SO(3)$, $u,v \in T_xSO(3)$, and $\xi, \eta \in \mathfrak{so}(3)$ with $\xi = x^{-1}u, \eta = x^{-1}v$. Then, the corresponding geodesic distance reduces to 
\begin{equation*}
    d(x,y)^2 = \|\Log(x^\top y)\|^2 =-\frac{1}{2} \tr\left[\Log (x^\top y)^2\right].
\end{equation*}

In the following, two scenarios are in order where we randomly initialize agents on $SO(3)$.

\subsection{Scenario 1 (Comparison of Consensus Algorithms)}
In this section, we run simulations comparing our solver to three other consensus algorithms. We randomly initialized $N=10$ agents $\x \in \C$ and run each algorithm with the same initial conditions in order to compare them. We choose the following two metrics for comparison: the \textit{consensus error} \cref{consensus-error}, and the \textit{\ac{rcm} Error} $\textstyle E_{\mathrm{RCM}}(\x) = \sum_{i=1}^N \mathrm{d}(x_i, \bar{z})^2$ which aggregates the error of each point $x_1,\ldots,x_N$ from the \ac{rcm} $\bar{z}$. Using these two metrics, we compare our solver against three other algorithms described in the following and illustrate the results in \Cref{fig:solvers}. For the technical details of these algorithm and their implementation see the extended version of this work \cite{Kraisler???} and the code \cite{Kraisler_Github}. 

\subsubsection{Algorithm \ref{algo:distributed-solver}}
To implement the continuous dynamics of our algorithm, we consider their forward-Euler discretization on each tangent space with step size $\epsilon=0.1$ and use the Lie group exponential mapping as our choice of retraction.

\subsubsection{Riemannian Consensus Algorithm}
As stated before, \cite{tron2012riemannian} was one of the first works to introduce a consensus algorithm for arbitrary Riemannian manifolds, referred to as ``R. Tron et. al'' in \Cref{fig:solvers}.

\subsubsection{Penalty method}
The penalty approach is a commonly used solver for constrained optimization problems. We implemented \cite[Algorithm 14.3.1]{conn2000trust}, while introducing a distributed implementation for it. In particular, we chose penalty parameter $\mu(s) = \frac{1}{\sqrt{s}}$, number of gradient descent iterations $S=50$, and step size $\epsilon = 0.1$.

\subsubsection{Lagrangian method}
A Lagrangian method is an approach to solving constrained optimization problems that involves finding saddle points of the Lagrangian function. We implemented the solver described in \cite[4.4.1]{bertsekas2014constrained}. 

It can be seen in \Cref{fig:solvers} that both error metrics for our algorithm are vanishing with linear rate. Also, note in \Cref{fig:solvers} that Riemannian Consensus Algorithm has a seemingly linear rate of decrease for the consensus error, similar to our algorithm but with faster rate. However, we emphasize that this algorithm does not converge to the \ac{rcm} and it only synchronizes the agents to a point. This is evident in \Cref{fig:solvers}, where the \ac{rcm} Error for this algorithm stops decreasing. Finally, as illustrated in \Cref{fig:solvers}, the Penalty and Lagrangian method have a sub-linear rate of decrease for both the consensus error and the \ac{rcm} error.


\begin{figure}[tp]
    \centering
    \includegraphics[trim={0 0 0 0.8cm},clip, width=0.9\columnwidth]{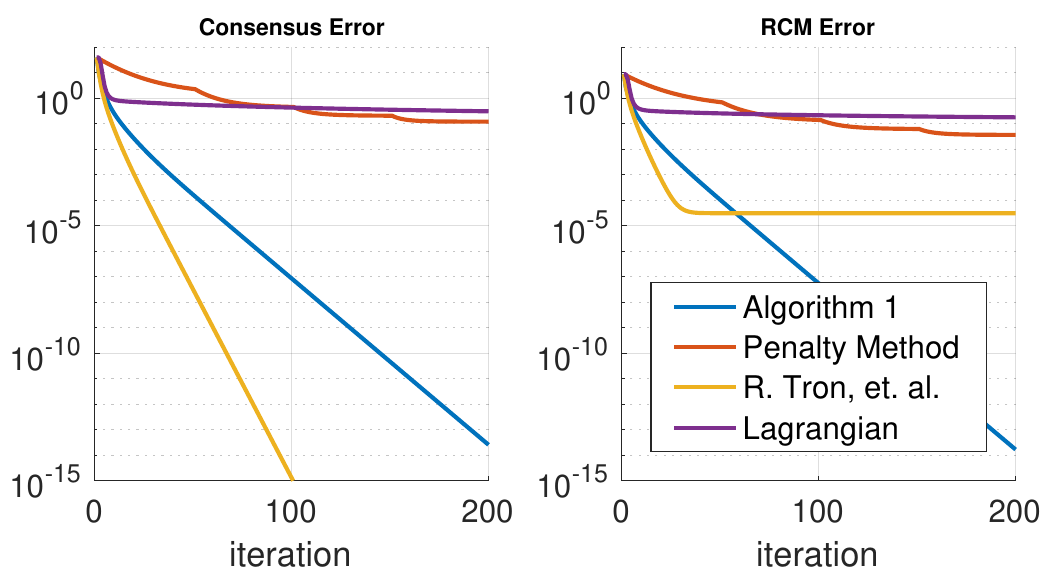}
    \caption{The consensus error (left) and the \ac{rcm} Error (right) at each iteration, comparing our algorithm (in blue) with three other algorithms.}
    \label{fig:solvers}
\end{figure}

\subsection{Scenario 2 (Linear rate of convergence)}
In this scenario,  we generate $100$ different problem instances with $N=10$ agents randomly initialized on $\C \subset SO(3)^{10}$. We run the algorithm on each instance for $200$ iterations for visualization of the converge rate. We illustrate the result in \Cref{fig:my_boxplot} showing the statistics of the RCM Error at each time step, confirming a linear rate. We note that although convergences analysis of the algorithm has only been provided for the Euclidean case, these numerical experiments--using different problem parameters--demonstrate the effectiveness of the proposed approach for distributed \ac{rcm} consensus on Lie groups.
\begin{figure}[tp]
    \centering
    \includegraphics[trim={0 0 0 1cm},clip,width=0.85\columnwidth]{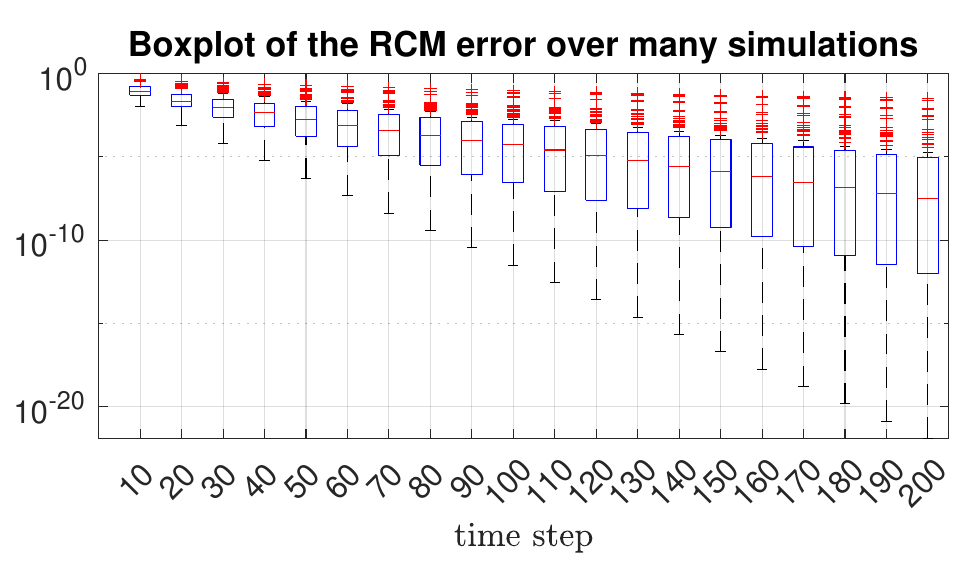}
    \caption{The RCM error of our proposed algorithm over $100$ randomly sampled problem instances.}
    \label{fig:my_boxplot}
    \vspace{-0.5cm}
\end{figure}

\section{Conclusions}\label{sec:conclusion}
In this work, we reformulate the \acf{rcm} as a distributed optimization problem and propose a novel and fast distributed solver on bi-invariant Lie groups. This in turn provides the first (completely) distributed solver for the \ac{rcm} consensus problem with no subroutines. The key idea behind the proposed algorithm is a correspondence between a consensus protocol on the Lie group and a dynamic consensus protocol on its Lie algebra. We have shown the properties of limit points of our algorithm in the general setting, and guaranteed its convergence in the Euclidean setting; the convergence analysis for the general case is the subject of our ongoing work.

\bibliographystyle{ieeetr}

\bibliography{citations}

\end{document}